\documentclass[letterpaper]{amsart} 

\usepackage{amsmath}
\usepackage[all]{xy}

\usepackage{psfrag}
\usepackage{epsfig}
 
\numberwithin{equation}{section}

\newtheorem{theorem}{Theorem}[section] 
\newtheorem{proposition}[theorem]{Proposition}

\newtheorem{lemma}[theorem]{Lemma} 
 
\theoremstyle{definition}

\newtheorem{definition}[theorem]{Definition}

\def\ZZ{\mathbb{Z}}

\begin{document}

\title{Reflection group relations arising from cluster algebras}

\author{Ahmet I. Seven}

\address{Middle East Technical University, Department of Mathematics, 06800, Ankara, Turkey}
\email{aseven@metu.edu.tr}


\thanks{The author's research was supported in part by the Scientific and Technological Research Council of Turkey (TUBITAK)}

\date{April 04, 2014}

\begin{abstract}
There is a well-known analogy between cluster algebras and Kac-moody algebras: roughly speaking, Kac-Moody algebras are associated with symmetrizable generalized Cartan matrices while cluster algebras correspond to skew-symmetrizable matrices. In this paper, we study an interplay between these two classes of matrices. We obtain relations in the Weyl groups of Kac-Moody algebras that come from mutation classes of skew-symmetrizable matrices. More precisely, we establish a set of relations satisfied by the reflections of the so-called companion bases; these include c-vectors, which parametrize coefficients in a cluster algebra with principal coefficients. These relations generalize the relations obtained by Barot and Marsh for finite type. For affine types, we also show that the reflections of the companion bases satisfy the relations obtained by Felikson and Tumarkin.  As an application, we obtain some combinatorial properties of the mutation classes of skew-symmetrizable matrices.   

\end{abstract}

\maketitle

\section{Introduction}
\label{sec:intro}

There is a well-known analogy between cluster algebras and Kac-moody algebras: roughly speaking, Kac-Moody algebras are associated with symmetrizable generalized Cartan matrices while cluster algebras correspond to skew-symmetrizable matrices. In this paper, we study an interplay between these two classes of matrices. We obtain relations in the Weyl groups of Kac-Moody algebras that come from mutation classes of skew-symmetrizable matrices. More precisely, we establish a set of relations satisfied by the reflections of the so-called companion bases; these include c-vectors, which parametrize coefficients in a cluster algebra with principal coefficients. These relations generalize the relations obtained by Barot and Marsh for finite type. For affine types, we also show that the reflections of the companion bases satisfy the relations obtained by Felikson and Tumarkin.  As an application, we obtain some combinatorial properties of the mutation classes of skew-symmetrizable matrices.


To state our results, we need some terminology. Let us recall that an integer matrix $B$ is skew-symmetrizable if $DB$ is skew-symmetric for some diagonal matrix $D$ with positive diagonal entries. 
For any matrix index $k$, the mutation of a skew-symmetrizable matrix $B$ at $k$ is another skew-symmetrizable matrix $\mu_k(B)=B'$: 
\begin{displaymath}
B' = \left\{ \begin{array}{ll}
B'_{i,j}=-B_{i,j} & \textrm{if $i=k$ or $j=k$}\\
B'_{i,j}=B_{i,j}+sgn(B_{i,k})[B_{i,k}B_{k,j}]_+ & \textrm{else}
\end{array} \right.
\end{displaymath}
(where we use the notation $[x]_+=max\{x,0\}$ and $sgn(x)=x/|x|$ with $sgn(0)=0$). 
Mutation is an involutive operation, so repeated mutations give rise to the \emph{mutation-equivalence} relation on skew-symmetrizable matrices. 
On the other hand, for any skew-symmetrizable $n\times n$ matrix $B$, a directed graph $\Gamma (B)$, called the diagram of $B$, is associated as follows: the vertices of $\Gamma (B)$ are the indices $1,2,...,n$ such that there is a directed edge from $i$ to $j$ if and only if $B_{ji} > 0$, and this edge is assigned the weight $|B_{ij}B_{ji}|\,$. Then the mutation $\mu_k$ can naturally be viewed as a transformation on diagrams (see Section~\ref{sec:pre} for a description). Note also that if $B$ is skew-symmetric then the diagram $\Gamma(B)$ may be viewed as a quiver and the corresponding mutation operation is also called quiver mutation.

Given the appearance of the mutation operation in many different areas of mathematics, it is natural to study properties of the mutation classes of skew-symmetrizable matrices and the associated diagrams. In this paper, we consider mutation classes of skew-symmetrizable matrices with an acyclic diagram; here a diagram is called acyclic if it does not have any oriented cycles at all. To study such mutation classes, a useful notion has been the notion of a \emph{quasi-Cartan companion}. More precisely a quasi-Cartan companion of $B$ is a symmetrizable matrix $A$ whose diagonal entries are equal to $2$ and whose off-diagonal entries differ from the corresponding entries of $B$ only by signs \cite{BGZ}. We will use a specific type of quasi-Cartan companions, which carry information about the the corresponding diagram.
To be more explicit, let us recall that, for a skew-symmetrizable matrix $B$, by a {subdiagram} of $\Gamma(B)$ we always mean a diagram obtained from $\Gamma(B)$ by taking an induced (full) directed subgraph on a subset of vertices and keeping all its edge weights the same as in $\Gamma(B)$. By a cycle in $\Gamma(B)$ we mean a subdiagram whose vertices can be labeled by elements of $\ZZ/m\ZZ$ so that the edges betweeen them are precisely $\{i,i+1\}$ for $i \in  \ZZ/m\ZZ$. We call a quasi-Cartan companion $A$ of a skew-symmetrizable matrix $B$ \emph{admissible} if it satisfies the following sign condition: for any cycle $Z$ in $\Gamma(B)$, the product $\prod _{\{i,j\}\in Z}(-A_{i,j})$ over all edges of $Z$ is negative if $Z$ is oriented and positive if $Z$ is non-oriented \cite{S3}. The main examples of admissible companions are the generalized Cartan matrices: if $\Gamma(B)$ is acyclic, i.e. has no oriented cycles at all, then the quasi-Cartan companion $A$ with $A_{i,j}= -|B_{i,j}|$, for all $i\ne j$, is admissible, furthermore (for $B$ with an acyclic diagram) 
\begin{equation}
\label{eq:sign-coherence}
\text{any matrix in the mutation class of $ B $ has an admissible quasi-Cartan companion.}
\end{equation}
This property has been shown in \cite{S6} for skew-symmetric matrices; for general skew-symmetrizable matrices it is conjectural. 

The admissible quasi-Cartan companions can be considered as Gram matrices of a bilinear form with respect to particular bases. To be more explicit, 
let us also recall that, for a skew-symmetrizable $B$ with an acylic diagram $\Gamma(B)$ and the corresponding generalized Cartan matrix $A$ with $A_{i,i}=2$ and $A_{i,j}=-|B_{i,j}|$ for $i\ne j$, there is an associated root system in the root lattice spanned by the simple roots $\alpha_i$  \cite{K}. For each simple root $\alpha_i$, the corresponding reflection $s_{\alpha_i}=s_i$ is the linear isomorphism defined on the basis of simple roots as $s_i(\alpha_j)=\alpha_j-A_{i,j}\alpha_i$. The group $W$ generated by these (simple) reflections is called the Weyl group of the root system. 
Then the real roots are defined as the vectors obtained from the simple roots by the action of $W$, i.e. by a sequence of reflections. It is well known that the coordinates of a real root with respect to the basis of simple roots are either all nonnegative or all nonpositive, see \cite{K} for details. For each real root $\alpha=w(\alpha_i)$, where $w\in W$, there is a corresponding reflection $s_\alpha=ws_iw^{-1}$. Following \cite{P,BM}, for a skew-symmetrizable matrix $B'$ which is mutation-equivalent to $B$ with a quasi-Cartan companion $A'$, we call a basis $\mathcal{B}=\{\beta_1,...,\beta_n\}$ of the root lattice a \emph{companion basis} for $ A' $ (or for $ B $) if each $\beta_i$ is a real root and $A'=(<\beta_j,{\beta}^{\vee}_i>)$,  the matrix of the pairings between the roots and the coroots. 
Under assumption \eqref{eq:sign-coherence}, each admissible quasi-Cartan companion of $B'$ has a companion basis; it can be obtained from the basis of simple roots by a sequence of mutations (\cite[Theorem 1.2]{S6}; see also Section \ref{sec:pre}). 

For a skew-symmetrizable matrix $B$,  we define, for vertices $i,j$ of $\Gamma(B)$, the following:
$$m_{ij}=
\begin{cases}
2 & \text{if $i$ and $j$ are not connected;} \\
3 & \text{if $i$ and $j$ are connected by an edge of weight $1$;} \\
4 & \text{if $i$ and $j$ are connected by an edge of weight $2$;} \\
6 & \text{if $i$ and $j$ are connected by an edge of weight $3$.} \\
\infty & \text{if $i$ and $j$ are connected by an edge of weight $\geq 4$.}
\end{cases}
$$
It is well-known that, if $\Gamma(B)$ is acyclic, then for the corresponding simple reflections $s_i$,  the order of $s_is_j$ in $W$ is equal to $m_{ij}$.
Our first main result gives relations satisfied by the reflections of a companion basis for $B$ whose diagram is not acyclic:



\begin{theorem}\label{th:invariance}
Suppose that $B_0$ is a skew-symmetrizable matrix whose diagram is acyclic and let $ W $ be the associated Weyl group. 
Suppose also that $B$ is mutation-equivalent to $B_0$ and let $A$ be an admissible quasi-Cartan companion of $B$. Let $\mathcal{B}=\{\beta_1,...,\beta_n\}$ be a companion basis for $A$. 
Then, under assumption \eqref{eq:sign-coherence}, the reflections $\{t_i:=s_{\beta_i}, i=1,...,n\}$ generate $W$ and satisfy the following relations:


\begin{enumerate}

\item[(i)]
$t_i^2=e$ for any vertex $i$ 

\item[(ii)]
For any two vertices $i,j$ in $\Gamma(B)$, $(t_it_j)^{m_{ij}}=e$ 

\item[(iii)] For any oriented cycle $C=\{1,...,d\}$ in $\Gamma(B)$ with 
$$1\xrightarrow{w_1} 2\xrightarrow{w_2} \cdots i-1\xrightarrow{w_{i-1}} i\xrightarrow{w_i} i+1\xrightarrow{w_{i+1}} \cdots \rightarrow {d}\xrightarrow{w_d} 1,$$
and any vertex $i$ in $C$ we have, 
writing $t_k=t_l$ if $k\equiv l$ $\mathrm{mod}$ $d$,
$$(t_it_{i+1}...t_{i+d-1}t_{i+d-2}...t_{i+1})^m=e$$ 
where, writing $q_j=\sqrt{w_j}$ and 
$x=(q_{i}...q_{i+d-2}-q_{i+d-1})^2$ with $q_k=q_l$ if $k\equiv l$ $\mathrm{mod}$ $d$, 
the number $m$ is the following: if $x=0$, then $m=2$; if $x=1$, then $m=3$; if $x=2$, then $m=4$; if $x=3$, then $m=6$; if $x\geq 4$, then $m=\infty$.


\end{enumerate}


\end{theorem}
\noindent
Let us note that these relations generalize those obtained in \cite{BM} for finite type. It has also been shown in \cite{BM} that, for finite type, these relations give a presentation of the corresponding finite Weyl group. 

We will establish a set of relations on more than one cycle as follows:

\begin{theorem}\label{th:invariance2}

In the set-up of Theorem~\ref{th:invariance},  let $i\xrightarrow{a} j$ be an edge in $\Gamma=\Gamma(B)$ and $p_1,...,p_r$ be an ordered set of directed paths from $j$ to $i$ such that
$$p_k: j\xrightarrow{w^k_0} j^k_1 \xrightarrow{w^k_1} \cdots  \rightarrow {j^k_{n_k}}\xrightarrow{w^k_{n_k}} i$$ and $\{j, j^k_1, \cdots,  {j^k_{n_k}}, i\}$ is an oriented (full, induced) cycle in $\Gamma$, for $ k=1,...,r $.
We assume that $p_k$'s are disjoint except at the end vertices $i,j$ (i.e. if $k\ne l$ then $j^k_x \ne j^l_y$; note that vertices on different $p_k$ may be connected to each other).
For each $k$, we denote $t^{(k)}=t_{j^k_1}  \cdots  t_{j^k_{n_k}}$; we denote the inverse of $t^{(k)}$ by $t^{(-k)}$.

Let $P$ denote the subdiagram on $p_1,...,p_r$. We say that a path $$q: j\xrightarrow{} j^{k_1}_{1}\leftrightarrow{} j^{k_2}_{l_2}\leftrightarrow{} \cdots  \leftrightarrow {j^{k_s}_{n_{k_s}}}\xrightarrow{} i$$ connecting $j$ to $i$ in $P$ is 
\emph{increasing} if we have $ 1\leq k_1\leq k_2 \leq .... \leq k_s \leq r $. 

Let $x=(\sqrt{a}-\sum \sqrt{w(q)})^2 $, where the sum is over all increasing paths $q$ from $j$ to $i$ in $P$.  Then 
$$(t_jt^{(1)}...t^{(r)}t_{i}t^{(-r)}...t^{(-1)})^m=e$$
where the number $m$ is the following: if $x=0$, then $m=2$; if $x=1$, then $m=3$; if $x=2$, then $m=4$; if $x=3$, then $m=6$; if $x\geq 4$, then $m=\infty$.



\end{theorem}

\begin{theorem}\label{th:invariance-affine}
Suppose that $B_0$ is a skew-symmetrizable matrix whose diagram is an extended Dynkin diagram. 
Suppose also that $B$ is mutation-equivalent to $B_0$ and let $A$ be an admissible quasi-Cartan companion of $B$. Let $\mathcal{B}=\{\beta_1,...,\beta_n\}$ be a companion basis for $A$. 
Then the reflections $t_i:=s_{\beta_i}$ satisfy the relations given in \cite[Table 4.1]{FTu}.
\end{theorem}

As an application of our study of the Weyl group in relation with the mutation operation, we obtain the following combinatorial property of the mutation classes.

\begin{theorem}\label{th:multiple-edge}
Suppose that $B$ is a skew-symmetrizable matrix such that $\Gamma(B)$ is acyclic. Suppose also that any edge in $\Gamma(B)$ has weight which is greater than or equal to $4$. Then, under the assumption \eqref{eq:sign-coherence}, for any $B'$ which is mutation-equivalent to $B$, any edge-weight in $\Gamma(B')$ is greater than or equal to $4$.
\end{theorem}






We prove our results in Section~\ref{sec:proof} after some preparation in Section~\ref{sec:pre}.

\vspace{.1in} 

\noindent
Acknowledgement: I thank M. Warkentin for drawing my attention to the property in Theorem~\ref{th:multiple-edge}.

\section{Preliminaries}
\label{sec:pre}

In this section, we will recall some more terminology and prove some statements that we will use to prove our results. First, let us recall that the diagram of a skew-symmetrizable (integer) matrix has the following property: 

\begin{align}
\label{eq:perfect-sq}
&\text{the product of weights along any cycle is a perfect square, i.e. the square}
\\
\nonumber
&\text{of an integer. }
\end{align}

\noindent
Thus we can use the term diagram to mean a directed graph, with no loops or two-cycles, such that the edges are weighted with positive integers satisfying \eqref{eq:perfect-sq}. Let us note that if an edge in a diagram has weight equal to one, then we do not specify its weight in the picture. 

For any vertex $k$ in a diagram $\Gamma$, the associated mutation $\mu_k$ changes $\Gamma$ as follows \cite{CAII}:
\begin{itemize} 
\item The orientations of all edges incident to~$k$ are reversed, 
their weights intact. 
\item 
For any vertices $i$ and $j$ which are connected in 
$\Gamma$ via a two-edge oriented path going through~$k$ (see  
Figure~\ref{fig:diagram-mutation-general}), 
the direction of the edge $\{i,j\}$ in $\mu_k(\Gamma)$ and its weight $\gamma'$ are uniquely determined by the rule 
\begin{equation} 
\label{eq:weight-relation-general} 
\pm\sqrt {\gamma} \pm\sqrt {\gamma'} = \sqrt {\alpha\beta} \,, 
\end{equation} 
where the sign before $\sqrt {\gamma}$ 
(resp., before $\sqrt {\gamma'}$) 
is ``$+$'' if $i,j,k$ form an oriented cycle 
in~$\Gamma$ (resp., in~$\mu_k(\Gamma)$), and is ``$-$'' otherwise. 
Here either $\gamma$ or $\gamma'$ can be equal to~$0$, which means that the corresponding edge is absent. 
 
\item 
The rest of the edges and their weights in $\Gamma$ 
remain unchanged. 
\end{itemize} 

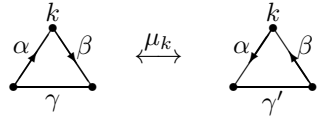
\begin{figure}[ht] 
\begin{center}
\setlength{\unitlength}{1.5pt} 
\begin{picture}(30,17)(-5,0) 
\put(0,0){\line(1,0){20}} 
\put(0,0){\line(2,3){10}} 
\put(0,0){\vector(2,3){6}} 
\put(10,15){\line(2,-3){10}} 
\put(10,15){\vector(2,-3){6}} 
\put(0,0){\circle*{2}} 
\put(20,0){\circle*{2}} 
\put(10,15){\circle*{2}} 
\put(2,10){\makebox(0,0){$\alpha$}} 
\put(18,10){\makebox(0,0){$\beta$}} 
\put(10,-4){\makebox(0,0){$\gamma$}} 
\put(10,19){\makebox(0,0){$k$}} 
\end{picture} 
$ 
\begin{array}{c} 
\stackrel{\textstyle\mu_k}{\longleftrightarrow} 
\\[.3in] 
\end{array} 
$ 
\setlength{\unitlength}{1.5pt} 
\begin{picture}(30,17)(-5,0) 
\put(0,0){\line(1,0){20}} 
\put(0,0){\line(2,3){10}} 
\put(10,15){\vector(-2,-3){6}} 
\put(10,15){\line(2,-3){10}} 
\put(20,0){\vector(-2,3){6}} 
\put(0,0){\circle*{2}} 
\put(20,0){\circle*{2}} 
\put(10,15){\circle*{2}} 
\put(2,10){\makebox(0,0){$\alpha$}} 
\put(18,10){\makebox(0,0){$\beta$}} 
\put(10,-4){\makebox(0,0){$\gamma'$}} 
\put(10,19){\makebox(0,0){$k$}} 
\end{picture} 
\end{center}
 
\vspace{-.2in} 
\caption{Diagram mutation} 
\label{fig:diagram-mutation-general} 
\end{figure}

\noindent 
This operation is involutive, i.e. $\mu_k(\mu_k(\Gamma))=\Gamma$, so it defines an equivalence relation on the set of all diagrams. More precisely, two diagrams are called \emph{mutation-equivalent} if they can be obtained from each other by applying a sequence of mutations. The \emph{mutation class} of a diagram $\Gamma$ is the set of all diagrams which are mutation-equivalent to $\Gamma$. If $B$ is a skew-symmetrizable matrix, then $\Gamma(\mu_k(B))=\mu_k(\Gamma(B))$ (see 
Section~\ref{sec:intro} for the definition of $\mu_k(B)$). Let us note that if $B$ is not skew-symmetric, then the diagram $\Gamma(B)$ does not determine $B$ as there could be several different skew-symmetrizable matrices whose diagrams are equal; however, if a skew-symmetrizing matrix $D$ is fixed, then $\Gamma(B)$ determines $B$, so mutation class of $\Gamma(B)$ determines that of $B$ (the matrix $\mu_k(B)$ shares the same skew-symmetrizing matrix $D$ with $B$ \cite{CAII}). 

Let us recall that an $n\times n$  matrix $A$ is called symmetrizable if there exists a diagonal matrix $D$ with positive diagonal entries such that $DA$ is symmetric. A crucial property of $A$ is sign symmetry: $sgn(A_{i,j})=sgn(A_{j,i})$. We say that $A$ is a quasi-Cartan matrix if it is symmetrizable and all of its diagonal entries are equal to $2$. 
A \emph{quasi-Cartan companion} (or "companion" for short) of a skew-symmetrizable matrix $B$ is a quasi-Cartan matrix $A$ with $|A_{i,j}|= |B_{i,j}|$ for all $i \ne j$. The basic example of a quasi-Cartan companion of $B$ is the associated generalized Cartan matrix $A$, which is defined as $A_{i,j}= -|B_{i,j}|$, for all $i\ne j$.

let us first discuss an extension of the mutation operation to quasi-Cartan companions:
\begin{definition} 
\label{def:comp-mut} 
Suppose that $B$ is a skew-symmetrizable matrix and let $A$ be a quasi-Cartan companion of $B$. 
Let $k$ be an index. For each sign $\epsilon=\pm1$, "the $\epsilon$-mutation of $A$ at $k$" is the quasi-Cartan matrix 
$\mu^\epsilon(A)=A'$ such that for any $i,j \ne k$: $A'_{i,k}=\epsilon sgn(B_{k,i})A_{i,k}$, $A'_{k,j}=\epsilon sgn(B_{k,j})A_{k,j}$, $A'_{i,j}=A_{i,j}-sgn(A_{i,k}A_{k,j})[B_{i,k}B_{k,j}]_+$. If $A$ is admissible, then $A'$ is a quasi-Cartan companion of $B'=\mu_k(B)$
\end{definition}

\noindent
Note that for $\epsilon=-1$, one obtains the formula in \cite[Proposition~3.2]{BGZ}. Also note that if $D$ is a skew-symetrizing matrix of $B$, then $D$ is also a symmetrizing matrix for $A$, with $DA=S$ symmetric. If we consider $S$ as the Gram matrix of a symmetric bilinear form on $\ZZ^n$ with respect to a basis $\mathcal{B}=\{e_1,...,e_n\}$, then $DA'=S'$ is the Gram matrix of the same symmetric bilinear form with respect to the basis $\mathcal{B'}=\{e'_1,e'_2,...,e'_n\}$ defined as follows: $e'_k=-e_k$; $e'_i=e_i-A_{k,i}e_k$ if $\epsilon B_{k,i}>0$; $e'_i=e_i$ if else. We write $\mu_k(\mathcal{B})=\mathcal{B}'$. Let us note that if $A$ is admissible and $\mathcal{B}$ is a companion basis, then $s_{e'_i}=s_{e_k}s_{e_i}s_{e_k}$, if $\epsilon B_{k,i}>0$ and $s_{e'_i}=s_{e_i}$ if else; this implies that the groups generated by the reflections of  $\mathcal{B}$ and $ \mathcal{B}' $ are equal.


The basic property of mutations of quasi-Cartan companions is the following: 
\begin{proposition}\label{prop:companion}
Suppose that $B_0$ is a skew-symmetrizable matrix whose diagram is acyclic and let $A_0$ be the corresponding generalized Cartan matrix. Suppose also that $B$ is mutation-equivalent to $B_0$. Then any admissible quasi-Cartan companion $A$ of $B$ can be obtained from $A_0$ by a sequence of mutations and/or simultaneous sign changes in some rows and columns. 

\end{proposition}
\noindent
These statements have been obtained in \cite{S6}; they follow from the basic fact that if $A$ is an admissible companion of $B$, then $\mu_k^\epsilon(A)=A'$  is an admissible companion of $\mu_k(B)$. 



We will use the following property of infinite Coxeter groups to prove one of our results:
\begin{proposition}\label{prop:multiple-edge-cox}
Let $W$ be a Coxeter group generated by simple reflections $s_1,...,s_n$ subject to the relations $s_i^2=1$ and $(s_is_j)^{m_{ij}}=1$ for $i\ne j$. Suppose that, for any $i\ne j$, we have $m_{ij}=2$ or $m_{ij}=\infty$.
Let $t_i\ne t_j$ be reflections in $W$, i.e. elements which are conjugate to simple reflections, and let $m'_{ij}$ be the order of the product $t_it_j$ in $W$. Then $m'_{ij}=2$ or $m'_{ij}=\infty$.

\end{proposition}

\begin{proof}
To prove the theorem we will use a geometric representation $\sigma:W\to GL({V})$ \cite[Chapter~5.3]{Hu}. To be more precise, let $V$ be the $n$-dimensional real vector space with basis ${\alpha}_1$, \dots ${\alpha}_n$ and equip $V$ with the symmetric bilinear form ${B}$ such that ${B}({\alpha}_i, {\alpha}_i)=2$ and ${B}({\alpha}_i, {\alpha}_j)=-2 \cos (\pi/m_{ij})$ for $i \neq j$. For each ${\alpha}_i$, the corresponding simple reflection $\sigma_i$ is defined on $V$ as $\sigma_i : v \mapsto v-{B}(v,{\alpha}_i) {\alpha}_i$; furthermore this action preserves the bilinear form ${B}$. The elements of $V$ that can be obtained from ${\alpha}_1$, \dots ${\alpha}_n$ by a sequence of reflections are called the real roots. 
The group $W$, which is generated by $\{s_1,...,s_n\}$, acts on $V$ by $\sigma(s_i)=\sigma_i$. Furthermore, this action is faithful \cite[Chapter~5.4]{Hu}, so, for any reflections $t,t'$ in $W$, the order of $\sigma(t)\sigma(t')$ is equal to the order of the product $tt'$ in $W$. 
Also, for any reflection $t$ in $W$, the transformation $\sigma(t)$ is a reflection with respect to a real root ${\beta}$,
i.e $\sigma(t) : v \mapsto v-{B}(v,\beta) {\beta}$
. 

Let us now suppose that $t_i \ne t_j$ are reflections in $W$ and let $\sigma(t_i)\ne \sigma(t_j)$ be reflections with respect to the real roots $\beta_i$ and $\beta_j$ respectively. To prove the theorem, we will show that the order of the product $\sigma(t_i)\sigma(t_j)$ is $2$ or $\infty$. 
Let us note that, under the assumptions of the theorem, we have ${B}({\alpha}_i, {\alpha}_j)\ne 0$ if and only if $|{B}({\alpha}_i,{\alpha}_j)|=2$; 
so ${B}({\beta}_i,{\beta}_j)=2k$ for some integer $k$. Then, by \cite[Lemma~2.77]{AB}, we have the following:
the order of $\sigma(t_i)\sigma(t_j)$ is $\infty$ if $k\ne 0$; it is $2$ if $k=0$. This completes the proof. 

\end{proof}

\section{Proofs of Main Results}
\label{sec:proof} Let us first recall some basics on root systems. In the setup of Theorems~\ref{th:invariance}, \ref{th:invariance2}, we denote by  $ ( , ) $ the invariant symmetric bilinear form defined by $A $ on the root lattice $ V=span(\beta_1,....,\beta_n) $ as $ (\beta_i,\beta_j) =d_iA_{i,j}=d_jA_{j,i}=(\beta_j,\beta_i)$ where $ d_i,d_j $ are the corresponding entries of the symmetrizing matrix $ D $, which is also a skew-symmetrizng matrix for $ B $ . This form is invariant under mutation of companions and simultaneous sign changes in rows and columns, thus it is the same form defined by the initial generalized Cartan matrix $ A_0 $ (Proposition \ref{prop:companion}). Let us note that
$s_{\beta_i}(\beta)=\beta-<\beta,{\beta_i}^{\vee}>\beta_i$, with $<\beta,{\beta_i}^{\vee}>= \dfrac{2(\beta_i,\beta)}{(\beta_i,\beta_i)} $. In particular, 
$s_{\beta_i}(\beta_j)=\beta_j-<\beta_j,{\beta_i}^{\vee}>\beta_i=\beta_j-A_{i,j}\beta_i$. Let us also note that  if $\{i,j\}$ is an edge in $\Gamma(B)$ whose weight $w$, then $ |A_{i,j}|=\sqrt{d_j/d_i}\sqrt{w}$  (because $ w= A_{i,j}A_{j,i}$ and $ d_iA_{i,j}=d_jA_{j,i} $) 

We will prove now some lemmas that we use to prove our theorems. 
\begin{lemma}\label{lem:edge2}
Let $ \{\beta_j,{\beta_i}\} $ be linearly independent real roots in the root system. defined by a symmetrizable generalized Cartan matrix $ A_0 $
Let $x= <\beta_j,{\beta_i}^{\vee}><\beta_i,{\beta_j}^{\vee}>$. 
Then the order $ m $ of $s_{\beta_i}s_{\beta_j}$ is as follows: if $x=0$, then $m=2$; if $x=1$, then $m=3$; if $x=2$, then $m=4$; if $x=3$, then $m=6$; if $x\geq 4$, then $m=\infty$.
\end{lemma}

\begin{proof}
Let $V'=span(\beta_i,\beta_j)$ in in the root lattice $ V $.  Let $ V'' $ be an orthogonal complement of $ V' $ in $ V $ with respect to the symmetric bilinear form $ ( , ) $ defined by $C=DA$ on $ V $, where $ D $ is the symmetrizing matrix.  Then the restriction of this form to $ V' $ is positive, in particular non-degenerate, also the reflections $s_{\beta_i},$ and $s_{\beta_j}$ act trivially on $ V'' $. Thus the order of $s_{\beta_i}s_{\beta_j}$ on $ V $ is equal to the order of $s_{\beta_i}s_{\beta_j}$ on $ V' $, which can be directly checked to be as claimed (recall that $s_{\beta_i}(\beta_j)=\beta_j-<\beta_j,{\beta_i}^{\vee}>\beta_i$, with $<\beta_j,{\beta_i}^{\vee}>= \dfrac{2(\beta_i,\beta_j)}{(\beta_i,\beta_i)} $   ).
\end{proof}

 
\begin{lemma}
\label{lem:lemma2}
In the setup of Theorem~\ref{th:invariance2}, applying simultaneous sign changes in rows and columns if necessary, we may assume that $A_{i,j}>0$ and, for any edge $\{u, v\} \ne \{i,j\} $ in $P$, we have $A_{u,v} \leq 0$.
\end{lemma}

\noindent 
Proof.  Applying a sign change at $ i $ or $ j $ if necessary, we may assume that  $A_{i,j}>0$ . Suppose that there is an edge $e: u \to v$ in $p_k$ such that
$A_{u,v}>0$, and assume without loss of generality that $e$ is closest such edge in $p_k$ to $j$.  Let $A'$ be the companion obtained by applying the sign change at $v$; if there is an edge
$e':u' \to v'$ in $p_k$ such that $A'_{u',v'}>0$ , then $e'$ is further from $j$ than $e$. Applying sign changes in the same way we have a companion $ A'' $ such that for any edge  $\{x, y\} \ne \{i,j\} $ in $p_k$, we have $A''_{x,y}<0$.  then we can rename $ A'' $ as $ A $, and assume for the rest of the proof that $A_{i,j}>0$ and, for any edge $\{u, v\} \ne \{i,j\} $ in $p_k$, we have $A_{u,v}<0$, for any $ k=1,...,r $.

 Suppose now that $P $ contains an edge $e: u \to v$  such that $ u=j^l_z \in p_l $, $ v=j^k_y \in p_k $ , $ k\ne l $ with $A_{u,v}>0$. We may assume without loss of generality that  if $x<y$ then $A_{u,j^k_x} \leq 0$. (otherwise we take $ v=j^k_x $, exchanging $ u $ and $ v $ if necessary to keep the orientation of $ e $). Let us first  suppose that $ u $ is connected to a vertex $ j^k_x \in p_k $ with  $1\leq x \leq y-1$; if $ x<y-1 $, we assume without loss of generality that $ u $ is not connected to $ j^k_{x+1},...,j^k_{y-1} $ 
 Then the cycle $C=\{ j^k_{x}, j^k_{x+1},...,,j^k_{y-1} ,j^k_{y},u\}$ is non-oriented, however it has only one edge, which is $ e $, such that the corresponding entry of $ A $ is positive,  contradicting the assumption that $A$ is admissible. Similarly if $ u $ is not connected to any vertex $ j^k_x \in p_k $ with  $1\leq x \leq y-1$, then $C=\{j, j^k_{1}, j^k_{2},...,j^k_{y},u,j^l_{z-1}, j^l_{z-2},...,j^l_{1}\}$ is non-oriented, however it has only one edge, which is $ e $, such that the corresponding entry of $ A $ is positive,  contradicting the assumption that $A$ is admissible. 

\vskip 0.3cm
\noindent
{\bf Proof of Theorem~\ref{th:invariance}.}
The statement that $\{t_i:=s_{\beta_i}, i=1,...,n\}$ generates $ W $ follows from basic properties of admissible quasi-Cartan companions and their mutations as discussed in Section \ref{sec:pre}. To be more explicit, let us suppose that $B_0=\mu_k...\mu_1(B) $. Then, applying sign changes if necessary, the matrix  $ A_0=\mu^\epsilon_k...\mu^\epsilon_1(A) $  is equal to the associated generalized Cartan matrix; let  $\mathcal{B'}=\{\beta'_1,...,\beta'_n\}$ be the basis obtained from $ \mathcal{B} $ by the same sequence of mutations and sign changes (Proposition \ref{prop:companion}). Then the groups generated by the reflections of $ \mathcal{B} $ and $ \mathcal{B'} $ are equal. On the other hand, the reflections of $ \mathcal{B'} $ satisfy exactly the same relations as the simple reflections of the root system, so the group generated by the reflections of $ \mathcal{B} $ is equal to $ W $, which is generated by the simple reflections.

Let us now show that the relations given in the statement hold. The relation (i) is well-known; (ii) follows from Lemma~\ref{lem:edge2}. As for the relation (iii), we assume, without loss of generality, that $i=1$. Let $B'=\mu_2...\mu_{d-2}\mu_{d-1}(B)$ and, e.g., for $\epsilon=-1$, let $\mathcal{B'}=\mu_2^\epsilon...\mu_{d-2}^\epsilon\mu_{d-1}^\epsilon(\mathcal{B})$. Then for $\beta_d'$ we have $s_{\beta'_d}=s_{\beta_2}...s_{\beta_{d-1}}s_{\beta_d}s_{\beta_{d-1}}...s_{\beta_2}$ and $\beta'_1=\beta_1$, furthermore the weight of the edge $\{1,d\}$ in $\Gamma(B')$ is equal to $(q_1...q_{d-1}-q_d)^2$. Thus, by Lemma~\ref{lem:edge2}, we have the relation (iii). 

\vskip 0.3cm

\noindent
{\bf Proof of Theorem~\ref{th:invariance2}.} 
Let $\beta=t^{(1)}...t^{(r)}(\beta_i)$. Note that
$\beta$ is the positive root whose reflection is $t^{(1)}...t^{(r)}t_{i}t^{(-r)}...t^{(-1)}$. 
Let us also note that 

(*) each reflection $t_{j^k_l}$ is applied only once in the product $t^{(1)}...t^{(r)}$. 

\noindent
For a vertex $u$ in $P$, we call a path 
$$q: u\leftrightarrow{} j^{k_1}_{l_1}\leftrightarrow{} j^{k_2}_{l_2}\leftrightarrow{} \cdots  \leftrightarrow {j^{k_s}_{n_{k_s}}}\xrightarrow{} i$$ connecting $u$ to $i$ in $P$ is \emph{increasing} if for $x<y$ we have $k_x \leq k_y$ (in the statement of the theorem this property was defined for the paths starting from $j$); the empty (stationary) path connecting $i$ to $i$ is also called increasing.

Let us now suppose that $u={j^k_l}$ is a vertex in $p_k$ (note that $ u\ne j $). 

\noindent
Claim: The coefficient of $\beta_u$ in $\beta$ is $\sqrt{d_i/d_u}$ times the sum of the square-roots of the weights of the increasing paths starting from $u$ and ending in $i$. 

To prove the claim, let us denote
$$T=t_{j^k_{l+1}}  \cdots  t_{j^k_{n_k}}t^{(k+1)}...t^{(r)}=t_{j^k_{l+1}}  \cdots  t_{j^k_{n_k}} (t_{j^{k+1}_1}  \cdots  t_{j^{k+1}_{n_{k+1}}})...(t_{j^r_1}  \cdots  t_{j^r_{n_r}})$$ 
Let us note that $\beta=t^{(1)}...t^{(k-1)}t_{j^k_{1}}  \cdots  t_{j^k_{l}}T(\beta_i)$.
Let us also note that the coefficient of $\beta_u$ in $\beta$ is equal to the coefficient of $\beta_u$ in $\delta:=t_{j^k_l}T (\beta_i)$ 
(as all $j^y_x$ are distinct by (*)), so to prove the claim it is enough to show that the coefficient of $\beta_u$ in 
$\delta$ is as in the statement of the claim. 
We will obtain this by an inductive argument on the number of reflections in $T$. For this purpose, let $\gamma:=T (\beta_i)$, so $\delta=t_u(\gamma)$. Let us assume that 

(**) each coefficient of $\beta_v$ in $\gamma$ is as in the statement of the claim, i.e. $\gamma=T (\beta_i)=\sum\sqrt{d_i/d_v} c_v\beta_v$, where
$c_v$ is equal to the sum of the square-roots of the weights of the increasing path starting from $v$ and ending in $i$, with $ v\ne j $ in $ P $.

Note that $c_u=0$ because $ t_u $ is not used in $ T $. (Note that $ t_j $ is not used in $ T $ either.)
Let us note that $\delta=t_u(\gamma)=\sum \sqrt{d_i/d_v}c_vt_u(\beta_v)$; also $A_{u,v}=-\sqrt{d_v/d_u}\sqrt{w}$ (Lemma \ref{lem:lemma2}), where $w$ is the weight of the edge $e$ connecting $u$ to $v$ (it is zero if there is no such edge), so $t_u(\beta_v)=\beta_v+A_{u,v}\beta_u=v+\sqrt{d_v/d_u}\sqrt{w}\beta_u$ (note that $ A_{u,v}<0 $ by Lemma~\ref{lem:lemma2} as $ \{u,v\}\ne \{i,j\} $).
Then the coefficient of $\beta_u$ in $\delta$ is equal to $\sum \sqrt{d_i/d_v}c_v\sqrt{d_v/d_u}\sqrt{w}=\sum \sqrt{d_i/d_u}\sqrt{w}c_v$, where each summand corresponds to the increasing paths obtained by composing $e$ with the paths for $v$ and $\sqrt{w}c_v$ is the square root of the weight of such a path;  they are all of the increasing paths from $u$. 
Then by induction on the number of reflections in $T$, the coefficient of $\beta_u$ in $\delta$, so in $\beta$, is equal to the number stated in the claim.

To show that the relations in the statement of the theorem hold, we will show that $ <\beta_j,{\beta}^{\vee}><\beta,{\beta}^{\vee}_j> = x$ as in the statement of the theorem  (Lemma \ref{lem:edge2}). For this, let us recall that $<\beta_j,{\beta}^{\vee}>= \dfrac{2(\beta,\beta_j)}{(\beta,\beta)} $ and $<\beta,{\beta}^{\vee}_j> =\dfrac{2(\beta,\beta_j)}{(\beta_j,\beta_j)}$ where $ ( , ) $ is the invariant symmetric bilinear form defined by $A $ on the root lattice (see the beginning of this section). 

To compute $ (\beta,\beta_j)=(\beta_j,\beta) $, 
recall that $\beta=\sum\sqrt{d_i/d_u} c_u\beta_u$, where $ c_u $ is the sum of the square roots of the weights of the increasing paths starting from $u$ and ending in $i$ (by the claim above), $u\ne j  $. 
If there is an edge between $ j$  and $u $, we denote its weight by $ w_u $; otherwise we take $ w_u=0 $. 
Let us note that  $ A_{j,u}=-\sqrt{d_u/d_j}\sqrt{w}$ and $ A_{j,i}=\sqrt{d_i/d_j}\sqrt{a}$ by Lemma~\ref{lem:lemma2} (as $ \{u,v\}\ne \{i,j\} $),  so
$(\beta_j,\beta)=\sum\sqrt{d_i/d_u} c_u(\beta_j,\beta_u)=\sqrt{d_i/d_i} c_id_j\sqrt{d_i/d_j}\sqrt{a}-\sum\sqrt{d_i/d_v}c_v\sqrt{d_j}\sqrt{d_v}\sqrt{w_v}$ where the sum is over all $ v\ne i,j $. Let us note that $ c_i=1 $, so  $(\beta_j,\beta) =\sqrt{d_j}\sqrt{d_i}(\sqrt{a}-\sum c_v\sqrt{w_v})$ where the sum is over over all $ v\ne i,j $. Note that each non-zero summand $ c_v\sqrt{w_v} $ is the sum of the square root of the weights of the increasing paths obtained by composing the edge $j \longrightarrow v$ with the paths for $v$ and these are all of the increasing paths from $j$, so 
 $\sum c_v\sqrt{w_v} =\sum \sqrt{w(q)} $ is the sum of the square roots of the weights of the increasing paths $ q $ starting from $j$ and ending in $i$. 
Thus the order of the element in the statement is  $$ <\beta_j,{\beta}^{\vee}><\beta,{\beta}^{\vee}_j> =\dfrac{2(\beta,\beta_j)}{(\beta,\beta)}\dfrac{2(\beta,\beta_j)}{(\beta_j,\beta_j)}=\dfrac{2(\beta,\beta_j)}{2d_i}\dfrac{2(\beta,\beta_j)}{2d_j}=
$$

$$=
\dfrac{(\beta,\beta_j)^2}{d_id_j}=\dfrac{(\sqrt{d_j}\sqrt{d_i}(\sqrt{a}-\sum \sqrt{w(q)}))^2}{d_id_j}=(\sqrt{a}-\sum \sqrt{w(q)})^2$$

\noindent
This completes the proof of the theorem.
\vskip 0.3cm

\noindent

\noindent
{\bf Proof of Theorem~\ref{th:multiple-edge}.} 
Let us suppose that $B'$ is mutation-equivalent to $B$ and $\Gamma(B')$ contains an edge $\{i,j\}$. Let $t_i$ and $t_j$ be the reflections corresponding to the companion basis elements $\beta_i,\beta_j$ (the companion basis can be obtained from the basis of simple roots by mutating accordingly). 
Let $m_{ij}$ be the order of $t_it_j$. Since the vertices $i$ and $j$ are connected in $\Gamma(B')$, we have $m_{ij}\ne 2$ (Lemma~\ref{lem:edge2}).
On the other hand, by Proposition~\ref{prop:multiple-edge-cox}, $m_{ij}$ is either $2$ or $\infty$; so $m_{ij}=\infty$. This completes the proof. 

\noindent
{\bf Proof of Theorem~\ref{th:invariance-affine}. } Let us first note that the existence of  an admissible companion for $ B $ has been shown in \cite{S3}.  Here we prove the theorem using a direct check. Let us show this for $ \tilde{D}_n  $. For this, suppose that $ \Gamma(B) $ contains a subdiagram $ X $ as in \cite[Table 4.1]{FTu}, labeled to be of type $ \tilde{D}_n $, with the same labeling of the vertices.
Applying sign changes if necessary, as in Lemma \ref{lem:edge2}, we may assume that $ A_{3,n+1}>0 $ and $ A_{i,j} < 0$ for all $ \{i,j\} \ne \{3,n+1\}$ in $ X $. 
Let us also note that $ t_1t_2t_3t_2t_1$ is the reflection with respect to the root $ \gamma=t_1t_2(\beta_3 )=\beta_1+\beta_2+\beta_3 $ and 
$ t_4t_5...t_nt_{n+1}t_n...t_5t_4$ is the reflection with respect to the root $ \delta=t_4t_5...t_n(\beta_{n+1} )=\beta_4+\beta_5+...+\beta_{n+1} $. An easy calculation shows that
 $ (\gamma,\delta)=0$, so $( t_1t_2t_3t_2t_1 t_4t_5...t_nt_{n+1}t_n...t_5t_4)^2=e$ as in \cite[Table 4.1]{FTu}. The other relations could be checked similarly.

\end{document}